\documentclass[11pt,reqno]{amsproc}
\usepackage[left=0.9in,top=0.9in,right=0.9in,bottom=0.9in]{geometry}
\usepackage{bbm}

\usepackage{amsfonts,amsmath,amsthm,amssymb,latexsym,mathrsfs,stmaryrd}
\usepackage{hyperref}
\usepackage{mathtools}
\usepackage{graphicx}
\usepackage{subcaption}

\usepackage{tikz}\usetikzlibrary{cd,fit,matrix,arrows,decorations.pathmorphing}
\tikzset{commutative diagrams/.cd}

\numberwithin{equation}{section}

\newtheorem{theorem}{Theorem}[section]
\newtheorem{corollary}[theorem]{Corollary}
\newtheorem{lemma}[theorem]{Lemma}
\newtheorem{proposition}[theorem]{Proposition}

\newtheorem{thmab}{Theorem}

\theoremstyle{definition}
\newtheorem{definition}[theorem]{Definition}
\newtheorem{definition-theorem}[theorem]{Definition-Theorem}
\newtheorem{definition-lemma}[theorem]{Definition-Lemma}

\newtheorem{notation}[theorem]{Notation}

\newtheorem{question}[theorem]{Question}

\newtheorem{remark}[theorem]{Remark}

\theoremstyle{remark}

\newtheorem*{remark*}{Remark}

\usepackage{enumitem}
\setenumerate[1]{label=\textup{(\alph*)}}
\setenumerate[2]{label=\textup{(\roman*)}}

\newcommand\Z{{\mathbb Z}}
\newcommand\N{{\mathbb N}}
\newcommand\Q{{\mathbb Q}}
\newcommand\C{{\mathbb C}}

\DeclareMathOperator{\im}{im}

\DeclareMathOperator{\codim}{codim}
\DeclareMathOperator{\Hom}{Hom}

\newcommand\subeq{\subseteq}

\newcommand\onto{\twoheadrightarrow}
\newcommand\incl{\hookrightarrow}
\newcommand{\map}[1][]{{\xrightarrow{#1}}} %\map[f] gives an arrow with f
\DeclarePairedDelimiter{\abs}{\lvert}{\rvert}

\DeclarePairedDelimiter{\set}{\{}{\}}
\DeclarePairedDelimiter{\pairing}{\langle}{\rangle}
\DeclarePairedDelimiter{\parens}{\lparen}{\rparen}

\DeclarePairedDelimiter\floor{\lfloor}{\rfloor}
\newcommand{\defn}{\textbf}

\newcommand{\YH}[1]{{{\color{blue}{#1}}}}

\DeclareMathOperator{\Conf}{Conf}
\DeclareMathOperator{\PConf}{PConf}

\newcommand{\spp}[1]{^{(#1)}}

\begin{document}

\title{Hilbert Series for Configuration Spaces of Punctured Surfaces}
\author{Yifeng Huang$^\dagger$}
\thanks{$^\dagger$University of Southern California, yifeng.huang\text{@}usc.edu}
\author{Eric Ramos$^{\dagger\dagger}$}
\thanks{$^{\dagger\dagger}$Stevens Institute of Technology, eramos3\text{@}stevens.edu}
\date{}
\begin{abstract}
Let $\Sigma_{g,r}$ denote the $r$-punctured closed Riemann surface of genus $g$. For every $g\geq 0$, we determine the four-variable generating function for the mixed Hodge numbers of the unordered configuration spaces of $\Sigma_{g,1}$. The cases where $g\geq 2$ are new. Combining a result of \cite{huang2020cohomology}, this determines the analogous generating function for $\Sigma_{g,r}$ for all $r\geq 1$. As an application of our formula we illustrate how classical homological stability results, as well as so-called secondary stability results of \cite{miller2019higher} can be interpolated to illustrate stable behaviors in the mixed Hodge numbers of these spaces which have been thus-far undiscovered.
\end{abstract}
\maketitle

%----article begins----
\section{Introduction}
For $n\geq 0$ and any smooth connected complex variety $\mathcal{M}$, let
\begin{equation}
\PConf_n(\mathcal{M}):=\set{(x_1,\dots,x_n)\in \mathcal{M}^n: x_i\neq x_j}
\end{equation}
be the ordered configuration space of $\mathcal{M}$, and
\begin{equation}
\Conf_n(\mathcal{M}):=\PConf_n(\mathcal{M})/S_n
\end{equation}
be the unordered configuration space of $\mathcal{M}$. As our primary focus in this work will be on the unordered configuration spaces, we will often just say configuration space to indicate unordered configuration space.

For $p,q,i\geq 0$, define the mixed Hodge number $h^{p,q;i}(\mathcal{M})$ as the dimension of the $(p,q)$-part of the mixed Hodge structure of $H^i(\mathcal{M};\Q)$. Our primary concern in this work will be with the generating function, 
\begin{equation}
f_\mathcal{M}(x,y,u,t):=\sum_{p,q,i,n\geq 0} (-1)^i h^{p,q;i}(\Conf_n(\mathcal{M}))\, x^p y^q u^i t^n,
\end{equation}
that records (up to sign) all mixed Hodge numbers of all configuration spaces of $\mathcal{M}$.

Generating functions similar to $f_\mathcal{M}$ have already seen a tremendous amount of attention in works such as \cite{cheonghuang2022betti,drummondcoleknudsen,vakilwood2015discriminants}, for instance. In keeping with these preceding works, a few words may be necessary to justify the decision to use a signed Hilbert series. Firstly, one should note that nothing is actually lost by making this choice since $\sum_{p,q,i,n\geq 0} h^{p,q;i}(\Conf_n(X))\, x^p y^q u^i t^n = f_X(x,y,-u,t)$. As for why we have chosen to do this, there is quite a bit of precedent going back to the seminal work of Macdonald \cite{macdonald1962poincare} and later work of Cheah \cite{Ch}. All of these previous works contribute to the following natural question, which we aim to resolve in part during the course of the present paper.

\begin{question}\label{mainQuestion}
Specialize $\mathcal{M}$ to be a smooth connected complex variety. Is $f_\mathcal{M}(x,y,u,t)$ always rational? If yes, what can we say about its denominator in general? For specific examples of $\mathcal{M}$, can we find an exact formula?
\end{question}

We note that Question \ref{mainQuestion} can be easily answered for $f_\mathcal{M}(x,y,1,t)$ because of the motivicity of the $E$-polynomial and a result of Vakil and Wood \cite[Proposition 5.9]{vakilwood2015discriminants} paired with some results of \cite{macdonald1962poincare}. On the other hand, Question \ref{mainQuestion} for $f_\mathcal{M}(1,1,u,t)$ concerns the rationality of the generating function for the Betti numbers of the configuration spaces of $\mathcal{M}$, which can be checked for any Riemann surface using \cite{drummondcoleknudsen}. The rationality question for $f_\mathcal{M}(1,1,u,t)$ for manifolds more general than Riemann surfaces remains open. Loosening our requirement that $\mathcal{M}$ be a manifold, and considering the (bivariate) generating function encoding the Betti numbers with varying homological index and number of points, it follows from \cite{ramos2018stability,an2019subdivisional} that the associated generating function is rational whenever $\mathcal{M}$ is a one-dimensional CW-complex (i.e. a graph).

We investigate $f_\mathcal{M}(x,y,u,t)$ as a common refinement of the extensively studied $f_\mathcal{M}(1,1,u,t)$ and the ``easy'' $f_\mathcal{M}(x,y,1,t)$. The primary result of this paper (Theorem \ref{thm:main}, below) will answer all of the above in the affirmative in the cases where $\mathcal{M}$ is a multi-punctured Riemann surface.

For $g\geq0 ,r\geq 1$, we denote by $X = \Sigma_{g,r}$ the $r$-punctured closed Riemann surface of genus $g$, equipped with any structure as a complex variety. We provide a clean explicit rational formula for $f_X(x,y,u,t)$. The case of $r=0$ is fundamentally different, and is treated by Pagaria in \cite{pagaria2020cohomology}. It is notable that our methods cannot be used to derive a closed form for the Hilbert Series in this case. Pagaria provides formulas \cite[Section 3]{pagaria2020cohomology}, which when pieced together form a rational expression for $f_X(x,x,u,t)$, though the resulting expression is not in any obvious way related to the ones given here for the $r \geq 1$ case.

\begin{thmab}
For $g\geq 0$ and $r\geq 1$, we have
\begin{equation}
f_{X}(x,y,u,t)=\frac{1}{(1+xyut)^{r-1}} \frac{\Phi_g\set{(1-xyz^2)(1-xz)^g(1-yz)^g}}{(1-t)(1-x^2 yu^2t^2)^g(1-xy^2u^2t^2)^g}, \label{mainThmFormula}
\end{equation}
where $\Phi_g$ is any $\Z[x,y]$-linear map satisfying
\begin{equation}
\Phi_g(z^j)=\begin{cases}
u^j t^j,& 0\leq j\leq g;\\
u^{j-1}t^j,& g+2\leq j\leq 2g+2.
\end{cases}
\end{equation}
\label{thm:main}
\end{thmab}

\begin{remark} \label{rmk:positiveNum}
    For $p+q=n$, the coefficient of $x^py^qz^n$ in $(1-xyz^2)(1-xz)^g(1-yz)^g$ is given precisely by
    \begin{align*}
        (-1)^n \parens*{\binom{g}{p}\binom{g}{q} -\binom{g}{p-1}\binom{g}{q-1}}.
    \end{align*}
    By the log-concavity of the binomial coefficeint sequence $\set{\binom{g}{j}}_j$ and the property that $\binom{g}{j}=\binom{g}{g-j}$, we have
    \begin{equation*}
        \mathrm{sgn}\parens*{\binom{g}{p}\binom{g}{q} -\binom{g}{p-1}\binom{g}{q-1}} = \begin{cases}
            +,& 0\leq p+q\leq g,\\
            0,& p+q=g+1,\\
            -,& g+1\leq p+q\leq 2g+2.
        \end{cases}
    \end{equation*}
    
    As a result, the $x^p y^q z^{g+1}$-coefficient of $(1-xyz^2)(1-xz)^g(1-yz)^g$ vanishes, so $\Phi_g(z^{g+1})$ need not be specified. Furthermore, $\Phi_g\set{(1-xyz^2)(1-xz)^g(1-yz)^g}$ has nonnegative coefficients in $x,y,-u,t$, so the coefficients of $f_X(x,y,-u,t)$ are nonnegative. While this positivity of $f_X(x,y,-u,t)$ is guaranteed by the fact that $h^{p,q;i}(\Conf_n(X))\geq 0$, the positivity of its \emph{numerator} $\Phi_g\set{(1-xyz^2)(1-xz)^g(1-yz)^g}$ is a stronger statement; in other words, \eqref{mainThmFormula} gives a cancellation-free formula to compute $h^{p,q;i}(\Conf_n(X))$.\\

    Relevant to this discussion, the aformentioned rational formula in the $r = 0$ case given by Pagaria \cite{pagaria2020cohomology} also has a manifestly positive numerator. It is the belief of the authors that this suggests a deeper topological or representation theoretic structure interacting with the mixed-Hodge numbers, such as the action of the mapping class group of $X$ \cite{bianchi2022mapping}.
\end{remark}

\begin{remark}
In light of \cite{huang2020cohomology}, it suffices to prove Theorem \ref{thm:main} for $r=1$. Indeed, in that work it is shown that the Hilbert Series in the $r$ and $r+1$ cases differ from one another precisely by a factor of $(1+xyut)$ in their respective denominators, so long as $r \geq 1$. For this reason, we will exclusively be assuming that $X = \Sigma_{g,1}$ is the once punctured Riemann surface of genus $g \geq 0$.

It should also be remarked that Theorem \ref{thm:main} in the case of $g = 1$ was shown in \cite{cheonghuang2022betti}. It is that work that originally noted that the formula \eqref{mainThmFormula} should hold in higher genus as well.
\end{remark}

\begin{remark}
The equation (\ref{mainThmFormula}) can also give one a sense of how things look when we also allow the parameters $g$ and $r$ to vary. For instance, if we were to expand $f_X(x,y,u,t)$ to a generating function in five variables by accounting for the parameter $r$, then the above immediately implies that this generating function would also be rational. The behavior in varying $g$ is a bit more obfuscated by the somewhat mysterious $\Phi_g$ function in the numerator.

The study of regular behaviors in configuration spaces that manifest when one varies not only the number of points being configured but also the underlying space, has become very common in recent years, though largely in the context of configuration spaces of graphs. See \cite{lutgehetmann2017representation,miyata2023graph,ramos2020application,knudsen2024robertson} for a small sampling of these kinds of results. Notably, in those contexts, it has been observed that the corresponding generating functions encoding Betti numbers with varying underlying graphs can be rational, combining \cite{ramos2020application} with \cite{nagel2021rationality}, or even algebraic \cite{ramos2021hilbert}. It would be interesting to see more results on generating functions associated with configuration spaces which account not only for the homological index and the number of points, but also variation in the base space.
\end{remark}

As an immediate application, substituting $x=y=1$, Theorem \ref{thm:main} provides a clean formula for the generating function for the Betti numbers up to sign. These Betti numbers were originally calculated in \cite{drummondcoleknudsen}. The necessity of the degree shift $\Phi_g$ to convert the numerator into a factored form is probably one reason why a simple formula for the numerator was not known before. 

As another application, setting $x,y,$ and $u$ all equal to 1 turns our main theorem into a statement about the generating function for the Euler characteristic of these spaces. Generating functions of this sort are the primary consideration of the famous theorem of Gal \cite{Gal}, which provided a beautifully clean rational description of the generating function for the Euler characteristic
\[
\sum_{n \geq 0} \chi(\Conf_n(S))t^n,
\]
whenever $S$ is a simplicial complex. The main result of this paper therefore recovers and expands on \cite{Gal} in the specific case of configuration spaces of punctured surfaces.

Perhaps the most mysterious seeming part of the formula in Theorem \ref{thm:main} is the combinatorial shifting operator $\Phi_g$. Originally, the necessary presence of such a shift was observed in an entirely ``brute force" fashion by the first author while he studied the Betti numbers found in \cite{drummondcoleknudsen}. Having proven Theorem \ref{thm:main}, however, we are now able to give a better explanation in terms of the geometry of the situation. As we will see in the proof of Proposition \ref{FinishingTouch}, the shift is a direct manifestation of the hard Lefschetz property of the operator $\omega\wedge:\bigwedge^*V\to \bigwedge^*V$, where $V=\Q^{2g}$ with basis $x_1,y_1,\dots,x_g,y_g$, and 
\begin{equation}
    \omega=x_1\wedge y_1+\dots+x_g\wedge y_g.
\end{equation}
(See Lemma~\ref{lem:hard-lefschetz}.) The reason why this operator shows up is because an isomorphic copy of it appears in the explicit description of the differential map of our spectral sequence, as explained in Section \ref{sec:mainProof}

To highlight the difficulty of Theorem \ref{thm:main}, we note that one of its consequences is that the mixed Hodge structure for $H^i(\Conf_n(\Sigma_{g,r}))$ is not pure for certain $i,n$ if $g,r > 1$. This is in contrast to the $g=r=1$ case, where one has purity of weight $\floor{3i/2}$ due to \cite{cheonghuang2022betti}. This purity was crucial in the determination of the $g=r=1$ case in \cite{cheonghuang2022betti}.%\YH{It can possibly be made a corollary if for each $g,r$, we can explicitly tell for what $i,n$ it is nonzero, and when it is pure. See Section 1.2 later.}

To illustrate the power of Theorem \ref{thm:main}, one can consider its implications in the asymptotic behavior in the parameters $i$ and $n$ of the Betti numbers, or mixed-Hodge numbers more generally. Classically speaking, homological stability in configuration spaces of manifolds considered the stable behavior of Betti numbers when $i$ was fixed and $n$ allowed to vary (See \cite{segal1973configuration}\cite{mcduff1975configuration}\cite{church2012homological} and the references therein). More recently it has been noted that regular behaviors in the Betti numbers also appear in ordered configuration spaces of orientable non-compact manifolds when one travels along any line of slope $\frac{1}{2}$, $i = \frac{1}{2}n$ \cite{miller2019higher}. Our Theorem \ref{thm:main} can be viewed as providing a stability theorem that is all encompassing, at least in the case of punctured Reimann surfaces. In other words, it's a version of stability for the Betti numbers that does not require one to look along linear slices of $i$ and $n$. Conversely, our main theorem implies the presence of ``good behavior" of the Betti numbers along any (quasi) linear relationship between $i$ and $n$. We will show the following during the course of this work as a direct corollary to our main theorem.

\begin{corollary}\label{cor:stability}
For $X=\Sigma_{g,r},$ $g\geq 0,$ $r\geq 1,$ we have
    \begin{enumerate}
        \item\label{stability} for all $p,q,i,n$, $h^{p,q;i}(\Conf_n(X))=0$ whenever $n<i$ and $h^{p,q;i}(\Conf_n(X))=h^{p,q;i}(\Conf_{n+1}(X))$ whenever $n \geq i+1$.
        \item\label{n=i} 
        \[ \sum_{p,q,i} (-1)^i h^{p,q;i}(\Conf_i(X)) x^p y^q z^i = \frac{T_{\leq g}\{(1 - xyz^2)(1 - xz)^g(1 - yz)^g\}}{(1+xyz)^{r-1}(1-x^2yz^2)^g(1-xy^2 z^2)^g},\]
        where the operator $T_{\leq g}$ means keeping only the terms with $z$-degree at most $g$.
        \item\label{n=i+1} \[ \sum_{p,q,i} (-1)^i h^{p,q;i}(\Conf_{i+1}(X)) x^p y^q z^i = \frac{\Psi_g\{(1 - xyz^2)(1 - xz)^g(1 - yz)^g\}}{(1+xyz)^{r-1}(1-x^2yz^2)^g(1-xy^2 z^2)^g},\]
        where $\Psi_g$ is any $\Q[x,y]$-linear map satisfying 
        \begin{equation}
        \Psi_g(z^j)=\begin{cases}
        z^j,& 0\leq j\leq g;\\
        z^{j-1},& g+2\leq j\leq 2g+2.
        \end{cases}
        \end{equation}
    \end{enumerate}
\end{corollary}

Breaking Corollary \ref{cor:stability} into its parts, \ref{stability} tells us that the mixed Hodge numbers of $\Conf_n(X)$ vanish above the line $i = n$. While this fact was already known (see \cite{miller2019higher}, for instance) our generating function explicitly observes the behavior by consequence of the fact that the denominator of (\ref{mainThmFormula}) is a univariate polynomial with coefficieints in $\Z[x,y]$ in the variable $ut$. Part~\ref{stability} also tells us that homological stability in homological index $i$ starts precisely at $n = i+1$. This is a vast improvement on the representation stable range for the spaces $\PConf_n(X)$, which is known to start at $n = 2i$.

The later two parts of Corollary \ref{cor:stability} tell us that if one were to look along either of the two linear strands $i = n$ or $i = n+1$, the generating function for the mixed Hodge numbers remains rational. Note that while it is always the case that rational multi-variate generating functions produce \emph{algebraic} generating functions along linear strands \cite[Chapter 6]{stanley1990enumerative}, our generating function (\ref{mainThmFormula}) is special in that these associated generating functions are rational.

One also sees that applying \ref{stability} together with \ref{n=i} and \ref{n=i+1}, one obtains rational formulas for the generating functions of the mixed Hodge numbers along any linear strand $i = k_1m, n = k_2m$.

Finally, combining \ref{stability} and \ref{n=i+1} also computes the stable mixed Hodge numbers $h^{p,q;i}(\Conf_\infty(X))$. We will prove Corollary \ref{cor:stability} in Section \ref{corProof}.

It is a simple algebraic computation to apply Corollary \ref{cor:stability} to recover the explicit formulas for the Betti numbers of configurations of punctured surfaces given by Drummond-Cole and Knudsen in \cite[Proposition 3.5]{drummondcoleknudsen}. We note that Corollary \ref{cor:stability} specialized to the Betti numbers by setting $x = y = 1$ and replacing $z$ with $-z$ tells us more than just a formula for these numbers. Indeed, as noted in Remark \ref{rmk:positiveNum}, it also reveals that the rational generating function for the Betti numbers has a numerator with positive coefficients!

To end this introduction, we take a moment to say a few words about what might be going on in the case of $\PConf_n(X)$. Much of the modern literature on ordered configuration spaces of manifolds and other spaces has put considerable stress on the fact that one must incorporate the natural actions of the symmetric groups into any computation one may perform on their cohomology groups. The language of Representation Stability (see \cite{church2013representation,church2015fi,miller2019higher} for a small sampling) was essentially designed to contend with the kinds of stable behaviors that one finds when studying the representations which appear in these cohomology groups. Parallel to these developments, there has also been considerable work trying to understand equivariant generalizations of Hilbert series in a number of circumstances wherein one has a family of spaces each carrying some kind of naturally compatible symmetry (see \cite{stapledon2011equivariant,elia2024techniques}, for some instances of this). It is our current belief, due to the fact that mixed Hodge structures on these ordered configuration spaces fully respect all permutation actions as well as the morphisms which introduce (or forget) points into the configuration space, that one can define an equivariant version of $f_X$, which carries a similar rationality property as in our main theorem. We leave this question as a possible interesting follow-up to the present work.

\section*{Notation}
%\YH{I noticed that some defs are in textbf while some in emph. Could you make definition fonts into $\defn{defn macro}$? We can agree on the style later just by editing the defn newcommand line (and I am ok with either).}

Here is a list of some general notation that will be used throughout this work. We collect these here for easy reference while reading.\\

\begin{itemize}
\item $[n]$ --- The set $\set{1,2,\dots,n}$.\\

\item $\Q \pairing{x_1,\dots,x_n}$ --- The $\Q$ vector space freely generated by $x_1,\dots,x_n$.\\

\item $\bigwedge[x_1,\dots,x_n]$, $\deg x_i=d_i$ --- free graded-commutative algebra over $\Q$ generated by homogeneous elements $x_i$ of degrees $d_i$. More precisely, if $\deg x_i$ are even and $\deg y_j$ are odd, then
\[ \bigwedge[x_1,\dots,x_n,y_1,\dots,y_m] = \Q[x_1,\dots,x_n]\otimes \bigwedge \Q\pairing{y_1,\dots,y_m},\]
with the latter factor referring to the exterior algebra of the vector space $\Q\pairing{y_1,\dots,y_m}$.\\

\item $X$ --- The once punctured smooth Riemann surface of genus $g\geq 0$, $\Sigma_{g,1}$.\\

\item $H^*(X)$ --- the rational cohomology ring of the space $X$. All (co)homology objects we study in this work will be assumed to be with rational coefficients.\\

\item $\alpha, \beta$ --- general elements of $H^1(X)$.\\

\item $\alpha_i, \beta_i$, where $i \in [n]$ --- the pullbacks of $\alpha$ and $\beta$, respectively, via the $i$-th projection $\pi_i: X^n\to X$ .\\

\item $h^{p,q;i}(X)$ --- the mixed Hodge numbers of $X$.\\

\item $\Conf_n(X)$ --- the $n$ point unordered configuration space of $X$. \\

\item $\PConf_n(X)$ --- the $n$ point ordered configuration space of $X$.\\

\item $f_X(x,y,u,t)$ --- the generating function $\sum_{p,q,i,n\geq 0} (-1)^i h^{p,q;i}(\Conf_n(X))\, x^p y^q u^i t^n$.\\

\item $g_{i,j}$ --- algebra generator for the 2nd page of the Totaro Spectral sequence computing the cohomology of $\PConf_n(X)$ (see \eqref{eq:E2Page})).\\

\item $I=(i_1,i_2,\dots,i_r)$ --- a tuple of distinct elements of $[n]$.\\

\item $g_I$ --- The product $g_{i_1i_2} g_{i_2i_3}\dots g_{i_{r-1}i_r}$.\\

\item $\alpha_I$ --- The product $g_{I}\alpha_{i_1} = g_{I}\alpha_{i_2} = \ldots = g_{I}\alpha_{i_r}$.\\

\end{itemize}

\section*{Acknowledgments}
ER was supported by DMS-2400460 and DMS-2137628. YH thanks Karthik Ganapathy for discussions leading to the formation of Corollary~\ref{cor:stability}.

\section{Computational background and setup}
\subsection{Computations in the cohomology of $\Sigma_{g,1}$}
Fix $g\geq 0$ and $X=\Sigma_{g,1}$. Then the cohomology ring of $X$ is given by
\[\
H^*(X)=\Q\pairing{1,x\spp{1},\dots,x\spp{g},y\spp{1},\dots,y\spp{g}},
\]
where $\deg x\spp{a}=\deg y\spp{a}=1$, $x\spp{a}$ has Hodge type $(1,0)$, and $y\spp{a}$ has Hodge type $(0,1)$ for each $1\leq a\leq g$.

For $n\geq 1$, the K\"unneth formula gives 
\[
H^*(X^n)=\bigwedge[x\spp{a}_i,y\spp{a}_i: 1\leq a\leq g, 1\leq i\leq n]/(\alpha_i\beta_i: \alpha,\beta\in H^1(X)),
\]
where for any $1\leq i\leq n$ and $\alpha\in H^1(X)$, we denote by $\alpha_i$ the pullback of $\alpha$ via the $i$-th projection $\pi_i: X^n\to X$. 

Let $\delta:X\to X^2$ be the diagonal map. We define the \textbf{diagonal class} $[\Delta]\in H^2(X^2)$ to be the image of $1$ under the Gysin map $\delta_*:H^0(X)\to H^2(X\times X)$, which we now take the time to explicitly construct.

Let $\Sigma=\Sigma_g$, and choose a volume form $\omega$ of $H^2(\Sigma)$. Let $x\spp{a}$ and $y\spp{a}$, where $1\leq a\leq g$, be chosen such that $x\spp{a}y\spp{a}=\omega$ and $x\spp{a}x\spp{b}=y\spp{a}y\spp{b}=x\spp{a}y\spp{b}=0$ in $H^2(\Sigma)$ for $a\neq b$. The compactly supported cohomology of $X$ is given by $H^0_c(X)=0$,  $H^1_c(X)=H^1(X)=H^1(\Sigma)$ and $H^2_c(X)=H^2(\Sigma)$. Under the above identification, the Poincar\'e duality pairing $H^1(X)\times H^1_c(X)\to H^2_c(X)$ is just the cup product $H^1(\Sigma)\times H^1(\Sigma)\to H^2(\Sigma)$. 

By K\"unneth's formula, the cohomology of $\Sigma^2$ is given by $H^*(\Sigma^2)=H^*(\Sigma)\otimes H^*(\Sigma)$. For $\gamma\in H^*(\Sigma)$, we denote $\gamma_1=\gamma\otimes 1$ and $\gamma_2=1\otimes \gamma$. There are natural maps $H^*_c(X^2)\incl H^*(\Sigma^2)\onto H^*(X^2)$. We have $H^4_c(X^2)\cong H^4(X^2)$, and let $\omega_1\omega_2$ be the volume form of $X^2$. As a vector subspace of $H^2(\Sigma^2)$, $H^2_c(X^2)$ is generated by $\alpha_1 \beta_2$, with $\alpha,\beta\in H^1(X)$. As a quotient space of $H^2(\Sigma^2)$, $H^2(X^2)$ is defined by the relation $\alpha_j \beta_j=0$ for $\alpha,\beta\in H^1(X)$ and $j=1,2$. The Poincar\'e pairing $H^2(X^2)\times H^2_c(X^2)\to H^4_c(X^2)$ descends from the usual cup product on $H^*(\Sigma^2)$; here, for the input in $H^2(X^2)$, we may choose an arbitrary lift in $H^2(\Sigma^2)$.

We are now ready to define the Gysin map. Denote by $\int$ the isomorphism that sends the volume form to $1$. In other words, $\int: H^2_c(X)\to \Q$ sends $\omega$ to $1$, and $\int:H^4_c(X^2)\to \Q$ sends $\omega_1\omega_2$ to $1$. Then the Gysin map $\delta_*:H^i(X)\to H^{i+2}(X^2)$ is uniquely characterized by
\begin{equation}
\int_X \epsilon \cdot \delta^*(\gamma) = \int_{X^2} \delta_*(\epsilon) \cdot \gamma, \text{ for all }\epsilon\in H^i(X)\text{, }\gamma\in H^{2-i}_c(X^2),
\end{equation}
where $\delta^*$ is the pullback $H^*_c(X^2)\to H^*_c(X)$ under the diagonal map, which is defined since $\delta$ is proper. We note that $\delta^*$ is induced from $H^*(X^2)\to H^*(X)$, which can be computed by $\delta^*(\gamma_j)=\gamma$ for all $\gamma\in H^*(X)$.

The above alternative construction for the Gysin map now gives us a means to define the diagonal class $[\Delta]$ by
\begin{equation}\label{eq:gysin}
\int_X \delta^*(\gamma) = \int_{X^2} [\Delta]\cdot \gamma, \text{ for all }\gamma\in H^2_c(X^2).
\end{equation}

Our next lemma uses the above characterization of the diagonal class to provide a slightly more useful expression for $[\Delta]$.

\begin{lemma}\label{lem:diagonal-class}
We have
\begin{equation}\label{eq:diagonal}
[\Delta]=-\sum_{a=1}^g \parens*{x\spp{a}_1 y\spp{a}_2+x\spp{a}_2 y\spp{a}_1}.
\end{equation}
\end{lemma}
\begin{proof}

It suffices to verify \eqref{eq:gysin} for $\gamma=x\spp{a}_1 x\spp{b}_2, y\spp{a}_1 y\spp{b}_2$, $x\spp{a}_1 y\spp{b}_2$, and $y\spp{a}_1 x\spp{b}_2$, assuming that $[\Delta]$ is given by \eqref{eq:diagonal}.

We demonstrate two examples only, and the rest are analogous. For $\gamma=x\spp{a}_1 x\spp{b}_2$, the left-hand side of \eqref{eq:gysin} is $\int_X x\spp{a}x\spp{b}=0$, and the right-hand side is 0 because a typical term is of the form
\begin{equation}
\int_{X^2} -x\spp{c}_1 y\spp{c}_2 x\spp{a}_1 x\spp{b}_2,
\end{equation}
which is zero because $x\spp{c}x\spp{a}=0$. As another example, let $\gamma=x\spp{a}_1 y\spp{a}_2$. Then the left-hand side of \eqref{eq:gysin} is $\int_X x\spp{a}y\spp{a}=1$, and the only nonvanishing term of the right-hand side is
\begin{equation}
\begin{aligned}
&\int_{X^2} -x\spp{a}_2 y\spp{a}_1 x\spp{a}_1 y\spp{a}_2 \\
= & \int_{X^2} x\spp{a}_1 y\spp{a}_1 x\spp{a}_2 y\spp{a}_2 = \int_{X^2}\omega_1\omega_2=1. \qedhere
\end{aligned}
\end{equation}
\end{proof}

\subsection{Distinguished generators for the cohomology of configuration space}

We begin this section by recalling the spectral sequence computing the cohomology of $\PConf_n(X)$ due to Totaro \cite{totaro1996configuration}. For $n\geq 1$, we define a bigraded\footnote{For graded commutativity, the degree of a homogeneous element of bidegree $(p,q)$ is $p+q$.} commutative algebra $E_2(X,n)$ by
\begin{align}
E_2(X,n)=\bigwedge[g_{ij}, x\spp{a}_i, y\spp{a}_i]/{\sim}, \label{eq:E2Page}
\end{align}
where the generators $g_{ij}$ are indexed by \textbf{unordered} pairs $\set{i,j}\subeq [n]$. The Hodge types and bidegrees of the generators are given by the following table.
\begin{center}
\begin{tabular}{l|l|l}
 & bidegree & Hodge type \\
 \hline
 $g_{ij}$ & $(0,1)$ & $(1,1)$ \\
 $x\spp{a}_i$ & $(1,0)$ & $(1,0)$\\
 $y\spp{a}_i$ & $(1,0)$ & $(0,1)$
\end{tabular}
\end{center}
The relations are given by
\begin{align}
g_{ij}g_{jk}+g_{jk}g_{ki}+g_{ki}g_{ij}&=0,& \set{i,j,k}\subeq [n],\label{eq:rel1}\\
g_{ij}\alpha_i&=g_{ij}\alpha_j, & \alpha\in H^1(X),\label{eq:rel2}\\
\alpha_i\beta_i&=0, & \alpha,\beta\in H^1(X).\label{eq:rel3}
\end{align}

Define a differential graded algebra structure on $E_2(X,n)$ by
\begin{equation}\label{eq:differential}
d\alpha_i=0, \quad dg_{ij}:=[\Delta]=-\sum_{a=1}^g \parens*{x\spp{a}_1 y\spp{a}_2+x\spp{a}_2 y\spp{a}_1},
\end{equation}
and define the graded commutative algebra $E_3(X,n)$ by taking the cohomology:
\begin{equation}\label{eq:cohomology}
E_3^{p,q}(X,n):=\frac{\ker(d: E_2^{p,q}\to E_2^{p+2,q-1})}{\im(d: E_2^{p-2,q+1}\to E_2^{p,q})}.
\end{equation}

Let the symmetric group $S_n$ act on $E_2(X,n)$ by permuting the lower indices. Note that $d$ is $S_n$-equivariant. Since we are working over $\Q$, taking $S_n$-invariants $(\cdot)^{S_n}$ is an exact functor, so $E_3(X,n)^{S_n}$ is the cohomology of $E_2(X,n)^{S_n}$ with respect to $d$. Since the mixed Hodge structure of $H^i(X)$ is pure of weight $i$ for any $i\geq 0$,\footnote{Importantly, this fails for $\Sigma_{g,r}$ if $r\geq 2$.} the same argument of \cite[\S 3]{cheonghuang2022betti} implies that
\begin{equation}
H^*(\Conf_n(X))\cong E_3(X,n)^{S_n}
\end{equation}
in a way that preserves the mixed Hodge numbers. Therefore, our computation of $f_X(x,y,u,t)$ reduces to computing the four-variable Hilbert series of $E_3(X,n)^{S_n}$ with respect to bigrading and the Hodge type. More precisely,
\begin{equation}\label{eq:precise}
    H^{w_1,w_2;i}(\Conf_n(X))\simeq \sum_{p+q=i} E_3^{(p,q,w_1,w_2)}(X,n)^{S_n},
\end{equation}
where $E_3^{(p,q,w_1,w_2)}$ denotes the part of $E_3^{p,q}$ with Hodge type $(w_1,w_2)$. 

Define the operator
\begin{equation}
e_n=\frac{1}{n!}\sum_{w\in S_n} w,
\end{equation}
then $E_2(X,n)^{S_n}=e_n(E_2(X,n))$. 

To explicitly determine $E_2(X,n)^{S_n}$ together with $d$, we need to describe a set of generators whose differentials are easy to compute. A priori, the generators are given by $e_n$ of monomials in $g_{ij}$ and $\alpha_i$, whose differentials could be complicated due to the Leibniz rule. To overcome this difficulty, we shall show that many monomials vanish in $E_2(X,n)$ after applying $e_n$. To this end, we recall some useful notation and lemmas from \cite{cheonghuang2022betti}.

\begin{notation}
An ordered tuple $I=(i_1,i_2,\dots,i_r)$ is always assumed to consist of distinct elements of $[n]$. For $r\geq 2$ and $\alpha\in H^1(X)$, define
\begin{equation}
g_I:=g_{i_1i_2} g_{i_2i_3}\dots g_{i_{r-1}i_r}
\end{equation}
and
\begin{equation}
\alpha_I:=g_I \alpha_{i_1}=\dots=g_I \alpha_{i_r},
\end{equation}
which takes advantage of the relation \eqref{eq:rel2}.
\end{notation}

\begin{lemma}[\cite{cheonghuang2022betti}]
For $I=(i_1,\dots,i_r)$ with $r\geq 3$, we have the following identities in $E_2(X,n)$:
\begin{enumerate}
\item $g_I g_{i_r i_1}=0;$
\item $e_n(g_I)=0;$
\item $e_n(\alpha_I)=0$ for $\alpha\in H^1(X)$;
\item Any monomial in $g_{ij}$'s is a linear combination of monomials of the form $g_{J_1}\dots g_{J_h}$ for disjoint $J_1,\dots,J_h$, where $h\geq 1$ and $\abs{J_1},\dots,\abs{J_h}\geq 2$. 
\end{enumerate}
\end{lemma}
\begin{proof}[Sketch of Proof]
The first two are well-known relations for the Orlik--Solomon algebra of the braid arrangement in $\C^n$. The third uses the (b) and the fact that $\alpha_I$ is equal to $g_I \alpha_{i_j}$ for any $1\leq j\leq r$. To prove (d), view a monomial in $g_{ij}$'s as a simple graph $G$ with vertex set $[n]$. By (a), we may assume $G$ is a forest. If some component of $G$ is not a chain, we apply \eqref{eq:rel1} and proceed by induction.
\end{proof}

The following is elementary but crucially used in \cite{cheonghuang2022betti}.

\begin{lemma}
We have $e_4(g_{12}g_{34})=0$ and $e_2(\alpha_i \alpha_j)=0$ for $\alpha\in H^1(X)$.
\end{lemma}

\begin{remark}
The proof uses the graded commutativity and that $\deg(g_{ij})=\deg(\alpha_i)=1$, an odd number. Importantly, the proof does not apply if we replace $g_{ij}$ by $\alpha_{ij}$ for $\alpha\in H^1(X)$, since $\alpha_{ij}$ is of even degree $2$.
\end{remark}

Using similar arguments in \cite{cheonghuang2022betti}, we conclude the following.

\begin{lemma}
As $\Q$-vector spaces, $E_2(X,n)^{S_n}$ is generated by elements of the form
\begin{equation}\label{eq:generators}
e_n\parens[\bigg]{g_{\bullet\bullet}^{\epsilon_0}(x\spp{1}_\bullet)^{\epsilon_1} \dots (x\spp{g}_\bullet)^{\epsilon_g} (y\spp{1}_\bullet)^{\epsilon_{g+1}} \dots (y\spp{g}_\bullet)^{\epsilon_{2g}} (x\spp{1}_{\bullet\bullet}\dots x\spp{1}_{\bullet\bullet})\dots (x\spp{g}_{\bullet\bullet}\dots x\spp{g}_{\bullet\bullet}) (y\spp{1}_{\bullet\bullet}\dots y\spp{1}_{\bullet\bullet})\dots (y\spp{g}_{\bullet\bullet}\dots y\spp{g}_{\bullet\bullet})
},
\end{equation}
where each $\epsilon_i$ is 0 or 1 (if 0, the corresponding factor does not appear, and we do not assign values to $\bullet$'s in it), $\bullet$'s are distinct indices in $[n]$, and each of the products $x\spp{a}_{\bullet\bullet}\dots x\spp{a}_{\bullet\bullet}$ and $y\spp{a}_{\bullet\bullet}\dots y\spp{a}_{\bullet\bullet}$ is allowed to be empty. \qed
\end{lemma}

\begin{remark}
Clearly, the element \eqref{eq:generators} does not depend on the exact assignments of the $\bullet$'s, as long as the number of $\bullet$'s is at most $n$. Therefore, the element \eqref{eq:generators} is determined by the combinatorial datum in Definition \ref{def:type} below.
\end{remark}

\begin{remark} \label{rmk:algebraTrick}
While we do not reproduce the proofs of \cite{cheonghuang2022betti}, we do take a moment to record a particular algebra trick that will continue to be useful to us going forward. Let us say there is an element $f\cdot f'\in E_2(X,n)$, where $f$ only involves the indices in some $J\subeq [n]$, and $f'$ only involves the indices in $J':=[n]\setminus J$. Suppose we have proven that $e_J(f)=0$, where
\begin{equation}
e_J=\frac{1}{\abs{J}!}\sum_{w\in S_{J}} w,
\end{equation}
where $S_J$ is the subgroup of $S_n$ consisting of permutations fixing $J'$. Then $e_J(f\cdot f')=e_J(f)\cdot f'=0$, and since $e_n$ factors through $e_J$, we have $e_n(f\cdot f')=0$. 
\end{remark}

The previous Lemma allows us to create an analog of $E_2(X,n)^{S_n}$ that is defined entirely combinatorially.\\

\begin{definition} \label{def:type}
    
Fix $g\geq 1$. Let $[g]:=\set{1,\dots,g}$ and let $\Gamma_{g}$ be the set of tuples $\mathbf{t}=(\rho,S,T,\vec u,\vec v)$ with
\begin{equation}
    \rho\in \set{0,1}, S,T\subeq [g], \vec u=(u_a)_{a\in [g]}, \vec v=(v_a)_{a\in [g]}, u_a,v_a\in \Z_{\geq 0}.
\end{equation}

Define a length function on $\Gamma_g$ by
\begin{equation}
    \ell(\mathbf{t}):=2\rho+\abs{S}+\abs{T}+2\abs{\vec u}+2\abs{\vec v},
\end{equation}
where $\abs{\vec u}=\sum_a u_a$ and $\abs{\vec v}=\sum_a v_a$. This induces a filtration on $\Gamma_g$:
\begin{equation}
    \Gamma_{g,0}\subeq \Gamma_{g,1}\subeq \dots
\end{equation}
where $\Gamma_{g,n}=\set{\mathbf{t}\in \Gamma_g: \ell(\mathbf{t})\leq n}$. Define grading functions $p,q,w_1,w_2$ on $\Gamma_g$ by
\begin{equation}\label{eq:grading}
    p(\mathbf{t})=\abs{S}+\abs{T}+\abs{\vec u}+\abs{\vec v}, \; q(\mathbf{t})=r+\abs{\vec u}+\abs{\vec v}, \; w_1(\mathbf{t})=r+\abs{S}+2\abs{\vec u}+\abs{\vec v}, \; w_2(\mathbf{t})=r+\abs{T}+\abs{\vec u}+2\abs{\vec v}.
\end{equation}
We thus get vector spaces $\Q \Gamma_g$ and $\Q \Gamma_{g,n}$ quadruply graded by $p,q,w_1,w_2$. We will interpret $(p,q)$ as the bidegree in the spectral sequence and $(w_1,w_2)$ as the Hodge type.
\end{definition}

In this notation, the distinguished generators $f_\mathbf{t}$ for $\mathbf{t}\in \Gamma_{g,n}$ generate $E_2(X,n)^{S_n}$. The differential of any $f_\mathbf{t}$ is easy to describe, thanks to the following observation.

\begin{lemma}\label{lem:diff-vanish}
For $\alpha\in H^1(X)$, we have $d\alpha_{ij}=0$. 
\end{lemma}
\begin{proof}
This is a direct calculation as a result of \eqref{eq:rel3} and \eqref{eq:differential}. We remark that the calculation fails if some $\alpha_i\beta_i$ is nonzero, which would occur if $X$ were the unpunctured $\Sigma_{g,0}$. 
\end{proof}

We may use Lemma \ref{lem:diff-vanish} to define a linear map $d:\Q\Gamma_g\to \Q\Gamma_g$ as follows. Let $\mathbf{t}=(r,S,T,\vec u,\vec v)\in \Q\Gamma_g$. If $r=0$, set $df_\mathbf{t}=0$. If $r=1$, set
\begin{equation}\label{eq:diff-abstract}
    df_{(1,S,T,\vec u,\vec v)}=-2\sum_{a\in [g]\setminus (S\cup T)} (-1)^{\abs{S}+\abs{T_{<a}}}f_{(0,S\sqcup a,T\sqcup a, \vec u,\vec v)},
\end{equation}
where $T_{<a}$ is the set of elements of $T$ that are $<a$. 

Note that $d$ induces a map $\Q\Gamma_{g,n}\to \Q\Gamma_{g,n}$. Moreover, the matrix coefficients of $d$ depend only on the $S$ and $T$ parts. More precisely, there are constants $c^{S,T}_{S',T'}$ defined for $S,T,S',T'\subeq [g]$ with $\abs{S'}+\abs{T'}=\abs{S}+\abs{T}+2$ such that
\begin{equation} \label{eq:TheC's}
    -\frac12 df_{(1,S,T,\vec u,\vec v)}=\sum_{\substack{S',T'\\\abs{S'}+\abs{T'}=\abs{S}+\abs{T}+2}} c^{S,T}_{S',T'}f_{(1,S',T',\vec u,\vec v)}.
\end{equation}

The crux of the proof of Theorem \ref{thm:main} is to show that the $f_\mathbf{t}$ form a basis for $E_2(X,n)^{S_n}$. In particular, that there is an isomorphism of $dg$-algebras $\Q\Gamma_{g,n} \cong E_2(X,n)$. From what has been discussed thus far, we already know that the elements $f_\mathbf{t}$ generate, and so it suffices to show that they are linearly independent of one another. This will be completed in the next section.

\section{The proof of Theorem \ref{thm:main}}

\subsection{The proof that the distinguished generators form a basis}

For any finite index set $J$, we have the notation $X^J=\prod_{j\in J}X$. K\"unneth gives $H^*(X^J)=H^*(X)^{\otimes J}$. Since $H^*(X)=H^0(X)\oplus H^1(X)$, we have the multi-degree decomposition

\begin{align*}
    H^*(X^J)=\bigoplus_{\mathbf{p}\in \set{0,1}^J} H^{\mathbf{p}}(X), \text{ where }\; H^{\mathbf{p}}(X):=\bigotimes_{j\in J} H^{\mathbf{p}_j}(X).
\end{align*}

Note that this same decomposition is compatible with any of the usual refinements of the degree decomposition $H^*(X)=H^0(X)\oplus H^1(X)$. In particular, it remains valid when one considers mixed Hodge numbers.  Since $H^*(X)=\Q\pairing{C}$ where
$C=\set{1,x\spp{1},\dots,x\spp{g},y\spp{1},\dots,y\spp{g}}$, we can write,

\begin{align*}
H^*(X)=\bigoplus_{c\in C} \mathbb{Q}\langle c \rangle = \bigoplus_{c\in C} H(X)|_{c},
\end{align*}

so that

\begin{align*}
H^*(X^J)=\bigoplus_{\mathbf{c}\in C^J} H(X)|_{\mathbf{c}}, \text{ where } H(X)|_{\mathbf{c}}:=\bigotimes_{j\in J} H(X)|_{\mathbf{c}_j}.
\end{align*}

We interpret $\mathbf{c}$ as a coloring of $J$ with colors from $C$. We have $\deg(c)=0$ if $c=1$ and $\deg(c)=1$ otherwise. Naturally, we define $\deg(\mathbf{c})=(\deg(\mathbf{c}_j))_{j\in J}\in \set{0,1}^J.$

Any (unordered) set partition $\pi = \set{B_1,\dots,B_{\abs{\pi}}}$ of $[n]$ defines a ``fat diagonal" in $X$, which we denote $X^{\pi}$,

\begin{align*}
X^\pi = \Hom(\pi,X) \simeq \set{(x_i)_{i\in [n]}: x_i=x_j\text{ if }i,j\in B_k, 1\leq k\leq \abs{\pi}}.
\end{align*}

Our next definition aligns the notion of partitioning the Cartesian product space $X^n$, with the notion of a $C$-coloring discussed above. 

\begin{definition}
    Let $\mathcal{P}$ be the set of set partitions of $[n]$. We define a \defn{$C$-colored partition} of $[n]$ to be a map $\mathbf{c}^\pi:\pi\to C$, where $\pi\in \mathcal{P}$. We write $\mathcal{CP}$ for the set of $C$-colored partitions. To each $\mathbf{c}^\pi$ we associate the composition map $\mathbf{c}:[n]\onto \pi \map[\mathbf{c}^\pi] C$, where the first map is the natural map that sends an element to the block it belongs to. 
\end{definition}

We will generally denote a $C$-colored partition by the pair $(\pi,\mathbf{c})$. Clearly, one may think of a $C$-colored partition as a set partition of $n$ whose every block has been assigned an element of $C$. According to this perspective, the map $\mathbf{c}$ records the $C$-color assigned to every individual element of $[n]$, whereas $\mathbf{c}^\pi$ accounts for the colors assigned to the blocks of $\pi$. We separate these two pieces of information in our notation, as we will ultimately need to call upon both of them in what follows.

We note that the symmetric group $S_n$ acts on $\mathcal{CP}$ on the right. To make this point clear, we observe that $\mathcal{P}$ has a more categorical definition: 
\[ \mathcal{P}=\set{[n]\onto \pi}/{\sim},\]
where $\pi$ is an abstract set, and the equivalence relation $\sim$ is generated by postcomposition by a bijection. Similarly, a $C$-colored partition of $[n]$ is an equivalence class of commutative triangles
\[
\begin{tikzcd}
\left[ n \right] \arrow[r, two heads] \arrow[dr,"\mathbf{c}"]
& \pi\arrow[d,"\mathbf{c}^\pi"] \\
& C
\end{tikzcd},
\]
where $\pi$ is a set, and the equivalence is up to composition with a bijection $\pi\map[\simeq]\pi'$. Note that such a composition is both a post-composition for the top map and a pre-composition for $\mathbf{c}^\pi$. Hence, the symmetric group $S_n$ acts on $\mathcal{CP}$ on the right by pre-composing the top map and $\mathbf{c}$.

For any $\mathbf{t}$, Our distinguished generator $f_\mathbf{t}$ will be related with a special class of $C$-colored partitions, defined as follows.

\begin{definition}
We say that a $C$-colored partition is of type $\mathbf{t} = (\rho,S,T,\vec{u},\vec{v})$ if it is compatible with $\mathbf{t}$ in the following sense:
\begin{itemize}
    \item if $\rho = 1$, then the partition has precisely one block of two elements that has been colored with $1 \in C$. Otherwise, if $\rho = 0$ then the partition has no block of size 2 colored with $1$;
    \item for each $i \in S$, the partition has a singleton block colored with $x^{(i)}$;
    \item for each $i \in T$, the partition has a singleton block colored with $y^{(i)}$;
    \item for each $1 \leq j \leq g$, the partition has $u_j$ size two blocks with color $x^{(j)}$;
    \item for each $1 \leq j \leq g$, the partition has $v_j$ size two blocks with color $x^{(j)}$;
    \item all remaining blocks in the partition are of size one, and colored with $1 \in C$.
\end{itemize}
We will write $\mathcal{CP_{\mathbf{t}}}$ for the set of $C$-colored partitions of type $\mathbf{t}$. Note that if $(\pi,\mathbf{c})\in \mathcal{CP}_\mathbf{t}$ for some $\mathbf{t}$, then $\pi$ can only contain singletons and pairs.
\end{definition}

Observe that, for any fixed $\mathbf{t}$, $\mathcal{CP}_\mathbf{t}$ is a single $S_n$ orbit and $\mathcal{CP}_\mathbf{t}\cap \mathcal{CP}_{\mathbf{t}'}=\varnothing$ if $\mathbf{t}\neq \mathbf{t}'$. Both of these facts will be relevant in what follows, as they essentially allow us to stratify $E_2^{*,*}(X,n)$ in a way that makes our desired linear independence of the distinguished generators more apparent. In particular, using the description of our spectral sequence given in \eqref{eq:E2Page},
\begin{align*}
   E_2^{*,*}(X,n)=\bigoplus_{( \pi,\mathbf{c})\in\mathcal{CP}} E_2^{\pi,\mathbf{c}}(X,n), \text{ where } E_2^{\pi,\mathbf{c}}(X,n):= H(X)|_{\mathbf{c}^\pi} \otimes A_\pi(\mathcal{A}),
\end{align*} 
where $A_\pi(\mathcal{A})$ is the Orlik--Soloman algebra of the flat within the braid arrangement determined by $\pi$, denoted by $F_\pi$. It is worth explictly pointing out that the definition of $E_2^{\pi,\mathbf{c}}(X,n)$ given above contains the term $H(X)|_{\mathbf{c}^\pi}$, where $\mathbf{c}^\pi$ is the coloring on \emph{blocks} of $\pi$ and not the coloring of $[n]$. More concretely, for a given $C$-colored partition $(\pi,\mathbf{c})$, we may construct monomials in $\bigwedge[g_{ij}, x\spp{a}_i, y\spp{a}_i]$ by selecting for all time a representative for each block of $\pi$, and,
\begin{itemize}
    \item for each representative $i_B$ selected above, appending $\alpha_{i_B}$ where $\alpha = \mathbf{c}^\pi(B)$, or $1$ if $\mathbf{c}^\pi(B) = 1$.
    \item For any collection of pairs $\{\{i_1,j_1\}\ldots \{i_l,j_l\}\}$, such that $l = \codim(F_\pi)$ and $F_\pi$ is equal to the intersection of the associated flats in the braid arrangement determined by the pairs, append the product $g_{i_1j_1}\cdots g_{i_lj_l}$.
\end{itemize}

To be clear, while the first step of the above process is determined entirely by the initial choice of representatives of the blocks of the partition $\pi$, one obtains different monomials by making different selections in the second step. For example, if $n=3, \pi=\set{\set{1,2,3}}$, and $\mathbf{c}^\pi=1$, then $g_{12}g_{13}$ and $g_{12}g_{23}$ are both monomials belonging to $E_2^{\pi,\mathbf{c}}(X,n)$. On the other hand, if $\pi = \set{\set{1,2},\set{3}}$ with $\mathbf{c}^\pi = (x\spp{2},y\spp{1})$, and $2$ as the chosen representative of the first block of $\pi$, then $x\spp{2}_2y\spp{1}_3g_{12}$ is the unique monomial in $E_2^{\pi,\mathbf{c}}(X,n)$.

We observe that relation $\eqref{eq:rel1}$ appears in $A_\pi(\mathcal{A})$, whereas relations \eqref{eq:rel2} and \eqref{eq:rel3} are baked into the monomials described above by the nature of a $C$-colored partition - one never has both $\alpha_i$ and $\beta_i$ appearing because $i$ only has a single color, and the first step of our coloring procedure picks a single $i$ for each block in $\pi$, avoiding the possibility of $\alpha_ig_{ij}$ and $\alpha_jg_{ij}$ both appearing. Returning to what was said above, the relation \eqref{eq:rel2} is essentially encoded in the fact that the right hand side of $E_2^{\pi,\mathbf{c}}(X,n)= H(X)|_{\mathbf{c}^\pi} \otimes A_\pi(\mathcal{A})$ is written in terms of $\mathbf{c}^\pi$ rather than $\mathbf{c}$. These facts justify our claim that $E_2^{*,*}(X,n)$ decomposes in this way.

This translation of a $C$-colored partition $(\pi,\mathbf{c})$ to a collection of monomials in $\bigwedge[g_{ij}, x\spp{a}_i, y\spp{a}_i]$ also illustrates the motivation behind the definition of compatibility with a tuple $\mathbf{t}$. If $(\pi,\mathbf{c}) \in \mathcal{CP}_\mathbf{t}$, and $m$ is a monomial appearing as described above in $E_2^{\pi,\mathbf{c}}(X,n)$, then this monomial can be rearranged to look like the term appearing inside the $e_n$ in the equation \eqref{eq:generators}. Moreover, because higher dimensionality of $E_2^{\pi,\mathbf{c}}(X,n)$ comes from the presence of choice in selecting the $g_{ij}$ terms (i.e. the second step in our above construction), $E_2^{\pi,\mathbf{c}}(X,n)$ must be one dimensional provided $(\pi,\mathbf{c})$ is of type $\mathbf{t}$, as every block has size 1 or 2.

Finally, observe that the above decomposition is compatible with the various symmetric group actions, in the sense of, for any $\eta \in S_n$,
\[
\eta \cdot E_2^{(\pi,\mathbf{c})} = E_2^{(\pi,\mathbf{c})\cdot \eta^{-1}}(X,n)
\]

\begin{proposition}
    The distinguished generators, $f_{\mathbf{t}}$, are linearly independent of one another.
\end{proposition}

\begin{proof}
Fix $\mathbf{t}$, and a monomial $m_0 \in \bigwedge[g_{ij}, x\spp{a}_i, y\spp{a}_i]$ in the form of the term inside $e_n$ in \eqref{eq:generators}. Note that it must be a nonzero element in some $E_2^{\pi_0,\mathbf{c}_0}(X,n)$ for a unique choice of $(\pi_0,\mathbf{c}_0)\in \mathcal{CP}_\mathbf{t}$. Then we first claim that $f_\mathbf{t}:=e_n(m_0)$ is nonzero. 

To start, the polynomial $e_n(m_0)$ is, by definition, comprised of the sum of all monomials obtained by applying permutations to $m_0$. To show that $e_n(m_0) \neq 0$ it will suffice to show that the component of $e_n(m_0)$ that lands inside of $E_2^{\pi_0,\mathbf{c}_0}(X,n)$ is necessarily non-zero. Using our translation between $C$-labeled partitions and monomials, we see that a permutation $\eta$ applied to $m_0$ can only land in $E_2^{\pi_0,\mathbf{c}_0}(X,n)$ if $\eta$,
\begin{enumerate}
    \item fixes all size 1 blocks of color not equal to 1, and
    \item permutes the size 2 blocks within any color.
\end{enumerate}
Critically, in so far as the monomials are concerned, permuting the size two blocks within a given color, say $\alpha$, amounts to permuting the terms $\alpha_{\bullet,\bullet}$ in $m_0$. It isn't hard to see, however, that these terms commute in $E_2^{*,*}(X,n)$ because they are of bi-degree $(1,1)$. In particular, each permutation $\eta$ which keeps $m_0$ within $E_2^{\pi_0,\mathbf{c}_0}(X,n)$ must do so by fixing it rather than sending $m_0$ to $-m_0$. This precludes the possibility of cancellation in $e_n(m_0)$.

Having proven our claim, we next observe that
\[
f_{\mathbf{t}} = e_n(m_0) \in \bigoplus_{(\pi,\mathbf{c}) \in \mathcal{CP}_{\mathbf{t}}} E_2^{\pi,\mathbf{c}}(X,n),
\]
by our observation above that $\mathcal{CP}_\mathbf{t}$ is closed under the action of $S_n$. Because each of the $\mathcal{CP}_\mathbf{t}$ are disjoint from one another, and therefore the vector spaces $\bigoplus_{(\pi,\mathbf{c}) \in \mathcal{CP}_{\mathbf{t}}} E_2^{\pi,\mathbf{c}}(X,n)$ (as we vary $\mathbf{t}$) are linearly disjoint, there cannot be total cancellation in any linear combination of the $f_{\mathbf{t}}$. The proof is concluded by the fact that we have shown $f_{\mathbf{t}} \neq 0$.
\end{proof}

As a result, we conclude that
\begin{corollary}\label{cor:substitution}
    For $X=\Sigma_{g,1}$, we have $(E_2(X,n)^{S_n},d)\simeq (\Q\Gamma_{g,n},d)$ in a way that preserves the quadruple grading $(p,q,w_1,w_2)$.
    \hfill\qedsymbol
\end{corollary}

\subsection{Concluding the proof using the distinguished basis}\label{sec:mainProof}

Having proven that our distinguished generators form a basis, we will now be able to prove our main theorem as a consequence. In particular, we compute dimensions of the components of the graded vector space $\Q\Gamma_g := \bigoplus_n \Q\Gamma_{g,n}$

Consider the exterior algebra $\Omega=\bigwedge[x_1,y_1,\dots,x_g,y_g]$ on $2g$ degree-one generators. The $i$-th graded part $\Omega^i$ has dimension $\binom{2g}{i}$ and is generated by basis elements $\alpha_{S,T}$ for $S,T\subeq [g], \abs{S}+\abs{T}=i$, where
\begin{equation}
    \alpha_{(S,T)}=\prod_{a\in S} x_a \prod_{b\in T} y_b,
\end{equation}
in which $a$'s and $b$'s are listed in ascending order. Let
\begin{equation}
    \omega=x_1\wedge y_1+\dots+x_g\wedge y_g,
\end{equation}
and consider the operator $\iota:\Omega^i\to \Omega^{i+2}$ defined by $\iota\alpha=\omega\wedge \alpha$. Then $\iota$ has matrix coefficients
\begin{equation}
    \iota\alpha_{(S,T)}=\sum_{\substack{S',T'\\\abs{S'}+\abs{T'}=\abs{S}+\abs{T}+2}} c^{S,T}_{S',T'} \alpha_{(S',T')}
\end{equation}
with the \emph{same} $c^{S,T}_{S',T'}$ as in \eqref{eq:TheC's}. The rank of $\iota$ is straight forward to compute, thanks to the following Lemma.

\begin{lemma}
    The operator $\iota:\Omega^i\to \Omega^{i+2}$ is injective if $i\leq g-1$ and surjective if $i\geq g-1$. 
    \label{lem:hard-lefschetz}
\end{lemma}
\begin{proof}
    This follows from the Hard Lefschetz Theorem for any abelian variety of dimension $g$.
\end{proof}

Returning to $(\Q\Gamma_g,d)$, from the definition \eqref{eq:diff-abstract} it is clear that $d^2=0$ and $d$ is homogeneous of degree $(p,q,w_1,w_2)=(2,-1,0,0)$, so we can define $(p,q,w_1,w_2)$-graded vector spaces
\begin{equation}\label{eq:hgn}
    H_g=\ker(d|\Q\Gamma_g)/\im(d|\Q\Gamma_g),\; H_{g,n}=\ker(d|\Q\Gamma_{g,n})/\im(d|\Q\Gamma_{g,n}).
\end{equation}
We now compute the Hilbert series of $H_{g,n}$ specialized in the way we want.

\begin{proposition}
    We have
    \begin{equation}\label{eq:hilbert-hgn}
        \sum_{p,q,w_1,w_2,n} \dim H_{g,n}^{(p,q,w_1,w_2)} x^{w_1}y^{w_2}(-u)^{p+q}t^n =\frac{\Phi_g\set{(1-xyz^2)(1-xz)^g(1-yz)^g}}{(1-t)(1-x^2 yu^2t^2)^g(1-xy^2u^2t^2)^g},
    \end{equation}
    where $\Phi_g$ is a $\Q[x,y]$-linear map defined by 
    \begin{equation}\label{eq:shiftDef}
    \Phi_g(z^j)=\begin{cases}
    u^j t^j,& 0\leq j\leq g;\\
    u^{j-1}t^j,& g+2\leq j\leq 2g+2.
    \end{cases}
    \end{equation}
\end{proposition}

\begin{proof}

Given $\rho\in \set{0,1}$, $0\leq \sigma,\tau\leq g$, $\vec u,\vec v\in \Z_{\geq 0}^g$, define $V_{(\rho,\sigma,\tau,\vec u,\vec v)}=\mathrm{Span}_\Q\set{f_{(\rho,S,T,\vec u,\vec v)}: \abs{S}=\sigma, \abs{T}=\tau}$. Consider
\begin{equation}
    H_{(1,\sigma,\tau,\vec u,\vec v)}=\ker(d:V_{(1,\sigma,\tau,\vec u,\vec v)}\to V_{(0,\sigma+1,\tau+1,\vec u,\vec v)})
\end{equation}
and
\begin{equation}
    H_{(0,\sigma,\tau,\vec u,\vec v)}=\frac{V_{(0,\sigma,\tau,\vec u,\vec v)}}{\im(d:V_{1,\sigma-1,\tau-1,\vec u,\vec v}\to V_{0,\sigma,\tau,\vec u,\vec v})}.
\end{equation}

Then
\begin{equation}
    H_{g,n}^{(p,q,w_1,w_2)}=\bigoplus_{\rho,\sigma,\tau,\vec u,\vec v}H_{(\rho,\sigma,\tau,\vec u,\vec v)}
\end{equation}
where the sum ranges over all $(\rho,\sigma,\tau,\vec u,\vec v)$ such that \eqref{eq:grading} holds and $2\rho+\sigma+\tau+\abs{\vec u}+\abs{\vec v}\leq n$. 

By Lemma \ref{lem:hard-lefschetz} and the fact that the same structure constants $c^{S,T}_{S',T'}$ appear in $\iota$ and \eqref{eq:TheC's}, we have
\begin{equation}
    \dim H_{(1,\sigma,\tau,\vec u,\vec v)}=
    \begin{cases}
        0, & \sigma+\tau\leq g-1,\\
        \binom{g}{\sigma}\binom{g}{\tau}-\binom{g}{\sigma+1}\binom{g}{\tau+1}, & \sigma+\tau\geq g-1,
    \end{cases}
\end{equation}
and
\begin{equation}
    \dim H_{(0,\sigma,\tau,\vec u,\vec v)}=
    \begin{cases}
        \binom{g}{\sigma}\binom{g}{\tau}-\binom{g}{\sigma-1}\binom{g}{\tau-1}, & \sigma+\tau\leq g+1\\
        0, &\sigma+\tau\geq g+1.
    \end{cases}
\end{equation}

Summing up the contribution to the Hilbert series, and noting that each $(\rho,\sigma,\tau,\vec u, \vec v)$ contributes to $H_{g,n}$ if and only if $2\rho+\sigma+\tau+\abs{\vec u}+\abs{\vec v}\leq n$, we get
\begin{align}
    &\sum_{p,q,w_1,w_2,n} \dim H_{g,n}^{(p,q,w_1,w_2)} x^{w_1}y^{w_2}(-u)^{p+q}t^n\\ 
    &= \sum_{\sigma,\tau,\vec u,\vec v} \dim H_{(1,\sigma,\tau,\abs{\vec u},\abs{\vec v)}} x^{1+\sigma+2\abs{\vec u}+\abs{\vec v}} y^{1+\tau+\abs{\vec u}+2\abs{\vec v}}(-u)^{1+\sigma+\tau+2\abs{\vec u}+2\abs{\vec v}}t^{2+\sigma+\tau+2\abs{\vec u}+2\abs{\vec v}}/(1-t)+\\
    &\sum_{\sigma,\tau,\vec u,\vec v} \dim H_{(0,\sigma,\tau,\vec u,\vec v)} x^{\sigma+2\abs{\vec u}+\abs{\vec v}} y^{\tau+\abs{\vec u}+2\abs{\vec v}}(-u)^{\sigma+\tau+2\abs{\vec u}+2\abs{\vec v}}t^{\sigma+\tau+2\abs{\vec u}+2\abs{\vec v}}/(1-t)\\
    &=\frac{1}{(1-t)(1-x^2 yu^2t^2)^g(1-xy^2u^2t^2)^g}\parens*{-xyut^2 P(-xut,-yut)+Q(-xut,-yut)},
\end{align}
where $$P(x,y):=\sum_{\substack{\sigma,\tau\geq 0\\\sigma+\tau\geq g-1}} x^\sigma y^\tau \parens*{\binom{g}{\sigma}\binom{g}{\tau}-\binom{g}{\sigma+1}\binom{g}{\tau+1}}$$ and $$Q(x,y):=\sum_{\substack{\sigma,\tau\geq 0\\\sigma+\tau\leq g}}x^\sigma y^\tau \parens*{\binom{g}{\sigma}\binom{g}{\tau}-\binom{g}{\sigma-1}\binom{g}{\tau-1}}$$
are the generating functions for $\dim H_{(1,\sigma,\tau,\abs{\vec u},\abs{\vec v})}$ and $\dim H_{(0,\sigma,\tau,\vec u,\vec v)}$, respectively.

The proof is complete once we prove the following lemma.
\end{proof}

\begin{lemma}\label{FinishingTouch}
    Let $P(x,y)$ and $Q(x,y)$ be as in the proof above. Then
    \begin{equation}\label{eq:finish}
        -xyut^2 P(-xut,-yut)+Q(-xut,-yut) = \Phi_g\set{(1-xyz^2)(1-xz)^g(1-yz)^g}
    \end{equation}
    where $\Phi_g$ is as in \eqref{eq:shiftDef} above.
\end{lemma}

\begin{proof}
    This is ultimately an exercise in expanding the relevant polynomials and performing simple algebraic manipulations. Instead of mechanically showing that the two sides of our desired equality are equal, we approach the problem from a more algebraic perspective. This argument hopefully stresses and elucidates how $\Phi_g$ arises in context. 

    To start, note that in the sums defining $P(x,y)$ and $Q(x,y)$, the constraint $\sigma,\tau\geq 0$ on the summands is unnecessary. This is obvious for $Q(x,y)$, whereas for $P(x,y)$, $\sigma$ being negative, for instance, forces $\tau$ to be larger than $g$, thereby killing the summand because of the $\binom{g}{\tau+1}$ factor.

    Now, we rewrite the binomial partial sums using the truncation operator: for $J\subeq \Z$ and a Laurent series $f(x,y)$, define
    \[ T^{x,y}_J f(x,y):=\sum_{\sigma+\tau\in J} x^\sigma y^\tau [x^\sigma y^\tau] f(x,y), \]
    where $[x^\sigma y^\tau]f(x,y)$ is the $x^\sigma y^\tau$-coefficient of $f(x,y)$. In particular, one can think of $T^{x,y}_Jf(x,y)$ as removing all terms in $f(x,y)$ whose total degree is outside of $J$.
    
    Then
    \begin{align*}
        P(x,y)&= T^{x,y}_{\geq g-1}\sum_{\sigma,\tau\in \Z} x^\sigma y^\tau \parens*{\binom{g}{\sigma}\binom{g}{\tau}-\binom{g}{\sigma+1}\binom{g}{\tau+1}} \\
        &= T^{x,y}_{\geq g-1} \parens*{(1+x)^g (1+y)^g - x^{-1}y^{-1}(1+x)^g (1+y)^g}\\
        &= T^{x,y}_{\geq g-1} \parens*{-x^{-1}y^{-1}(1-xy)(1+x)^g (1+y)^g}\\
        &= -x^{-1}y^{-1} T^{x,y}_{\geq g+1} \parens*{(1-xy)(1+x)^g (1+y)^g},
    \end{align*}
    and
    \begin{align*}
        Q(x,y)&= T^{x,y}_{\leq g}\sum_{\sigma,\tau\in \Z}x^\sigma y^\tau \parens*{\binom{g}{\sigma}\binom{g}{\tau}-\binom{g}{\sigma-1}\binom{g}{\tau-1}}\\
        &= T^{x,y}_{\leq g} \parens*{(1+x)^g (1+y)^g-xy(1+x)^g(1+y)^g}\\
        &=T^{x,y}_{\leq g} \parens*{(1-xy)(1+x)^g (1+y)^g}.
    \end{align*}

    Substituting to the left-hand side of \eqref{eq:finish}, and noting that the substitution $x\mapsto -xut, y\mapsto -yut$ does not affect the total degree in $x,y$, we get
    \begin{align*}
        &-xyut^2 P(-xut,-yut)+Q(-xut,-yut) \\
        &= \parens*{-xyut^2 (-(xut)^{-1}(yut)^{-1}) T^{x,y}_{\geq g+1}+T^{x,y}_{\leq g}} \parens*{(1-xyu^2t^2)(1-xut)^g (1-yut)^g}\\
        &= \parens*{u^{-1} T^{x,y}_{\geq g+1}+T^{x,y}_{\leq g}} \parens*{(1-xyu^2t^2)(1-xut)^g (1-yut)^g},
    \end{align*}
    which clearly agrees with $\Phi_g\set{(1-xyz^2)(1-xz)^g(1-yz)^g}.$
\end{proof}

\begin{proof}[Proof of Theorem~\ref{thm:main}]
    By \cite{huang2020cohomology}, it suffices to prove Theorem~\ref{thm:main} with $r=1$. Assume $r=1$ from now on. Combining \eqref{eq:precise}, the $S_n$-invariant version of \eqref{eq:cohomology} (see the discussion below it), Corollary~\ref{cor:substitution}, and \eqref{eq:hgn}, it follows that the left-hand side of \eqref{eq:hilbert-hgn} is precisely the left-hand side of \eqref{mainThmFormula} with $r=1$. Therefore, Theorem~\ref{thm:main} follows from \eqref{eq:hilbert-hgn}. 
\end{proof}

\section{The Proof of Corollary \ref{cor:stability}}\label{corProof}

In this final section, we expand out the algebra necessary to prove Corollary \ref{cor:stability}.

\begin{proof}[Proof of Corollary \ref{cor:stability}]
    Throughout the proof, we regard $\Z[x,y]$ as the coefficient ring, so polynomials in $x,y$ are treated as constants.
    
    By the definition of $\Phi_g$, there exist sequences $\set{a_i(x,y)}_{i\in \Z}$ and $\set{b_i(x,y)}_{i\in \Z}$  of polynomials in $\Z[x,y]$ such that
    \[ \Phi_g\set{(1-xyz^2)(1-xz)^g(1-yz)^g} = \sum_{i\in \Z} a_i u^i t^i + b_i u^i t^{i+1}.\]
    Note that $a_i=0$ unless $0\leq i\leq g$ and $b_i=0$ unless $g+1\leq i\leq 2g+1$.

    Since $1/((1+xyut)(1-x^2y u^2 t^2)^g(1-xy^2 u^2 t^2)^g)$ is a power series in $\Z[x,y][[ut]]$, there exist sequences $\set{c_i(x,y)}_{i\in \Z}$ and $\set{d_i(x,y)}_{i\in \Z}$ such that
    \[ \frac{1}{(1+xyut)^{r-1}} \frac{\Phi_g\set{(1-xyz^2)(1-xz)^g(1-yz)^g}}{(1-x^2 yu^2t^2)^g(1-xy^2u^2t^2)^g} = \sum_{i\in \Z} c_i u^i t^i + d_i u^i t^{i+1}.\]

    Multiplying by $1/(1-t)=1+t+t^2+\dots$, we conclude that the $u^i t^n$-coefficient of $f_X(x,y,u,t)$ is
    \[ [u^i t^n]f_X(x,y,u,t)=\begin{cases}
        0,&n<i;\\
        c_i(x,y),&n=i;\\
        c_i(x,y)+d_i(x,y),&n\geq i+1,
    \end{cases}\]
    proving the stability \ref{stability}.

    To prove \ref{n=i}, we extract
    \begin{align*}
        \sum_i z^i [u^i t^i]f_X(x,y,u,t) &= \sum_i z^i [u^i t^i]\set*{\frac{\sum_j a_j u^j t^j+b_j u^j t^{j+1}}{(1+xyut)(1-x^2y u^2 t^2)^g(1-xy^2 u^2 t^2)^g}\frac{1}{1-t}}\\
        &=\sum_i z^i [u^i t^i]\set*{\frac{\sum_j a_j u^j t^j}{(1+xyut)(1-x^2y u^2 t^2)^g(1-xy^2 u^2 t^2)^g}}\\
        &=\frac{\sum_i a_i z^i}{(1+xyz)(1-x^2y z^2)^g(1-xy^2 z^2)^g},
    \end{align*}
    where the second equality is because $b_j u^j t^{j+1}$ term cannot contribute and the only term in $1+t+t^2+\dots$ that contributes is the constant term $1$. Finally, we note that $\sum_i a_i z^i=T_{\leq g}\{(1 - xyz^2)(1 - xz)^g(1 - yz)^g\}$. This finishes the proof of \ref{n=i}.

    To prove \ref{n=i+1}, a similar inspection yields
    \begin{align*}
        \sum_i z^i [u^i t^{i+1}]f_X(x,y,u,t) &= \sum_i z^i [u^i t^{i+1}]\set*{\frac{\sum_j a_j u^j t^j+b_j u^j t^{j+1}}{(1+xyut)(1-x^2y u^2 t^2)^g(1-xy^2 u^2 t^2)^g}\frac{1}{1-t}}\\
        &=\sum_i z^i [u^i t^{i+1}]\set*{\frac{(\sum_j a_j u^j t^j)\cdot t+(\sum_j b_j u^j t^{j+1})\cdot 1}{(1+xyut)(1-x^2y u^2 t^2)^g(1-xy^2 u^2 t^2)^g}}\\
        &=\frac{\sum_i a_i z^i+\sum_i b_i z^i}{(1+xyz)(1-x^2y z^2)^g(1-xy^2 z^2)^g}.
    \end{align*}
    Again, we note that $\sum_i a_i z^i+\sum_i b_i z^i=\Psi_g\{(1 - xyz^2)(1 - xz)^g(1 - yz)^g\}$. This finishes the proof of \ref{n=i+1}.
\end{proof}

%----article ends----
\bibliography{ref.bib}
\bibliographystyle{plain}
%\bibliographystyle{abbrv}
%End of Article
\end{document}